\documentclass[11pt,twoside]{article}
\usepackage{amsfonts}
\usepackage{fancyhdr}
\usepackage{titlesec}
\usepackage{cite}
\usepackage{ifthen}
\usepackage{amssymb}
\usepackage{fancyhdr}
\usepackage{titlesec}
\usepackage[arrow,matrix]{xy}
\usepackage{pifont}
\usepackage{stmaryrd}
\usepackage{setspace}
\usepackage{indentfirst}
\usepackage{amsmath,amssymb,amscd,bbm,amsthm,mathrsfs,dsfont}
\input amssym.def
\titleformat{\section}{\centering\large\bfseries}{\S\arabic{section}}{1em}{}
\newboolean{first}
\setboolean{first}{true}

\newtheorem{theorem}{Theorem}[section]
\newtheorem{lemma}{Lemma}[section]
\newtheorem{rem}{Remark}[section]

\newtheorem{definition}{Definition}[section]
\begin{document}
\setlength\abovedisplayskip{2pt}
\setlength\abovedisplayshortskip{0pt}
\setlength\belowdisplayskip{2pt}
\setlength\belowdisplayshortskip{0pt}

\begin{center}
{\bf \LARGE A Generalization of Exponential Class and Its
Applications\footnote{Corresponding author: GAO Hongya, E-mail:
ghy@hbu.cn, TEL: 863125079658, FAX: 863125079638.}}

\vspace{3mm}

{\small \textsc{GAO Hongya}\quad
\textsc{LIU Chao} \quad  \textsc{TIAN Hong}\\}

{\small  College of Mathematics and Computer Science, Hebei
University, Baoding, 071002, China\\}
\end{center}

\begin{center}
\begin{minipage}{135mm}
{\bf \small Abstract}.\hskip 2mm {\small A function space,
$L^{\theta, \infty)}(\Omega)$, $0\le \theta <\infty$, is defined. It
is proved that $L^{\theta, \infty)}(\Omega)$ is a Banach space which
is a generalization of exponential class. An alternative definition
of $L^{\theta, \infty)}(\Omega)$ space is given. As an application,
we obtain weak monotonicity property for very weak solutions of
$\cal A$-harmonic equation  with variable coefficients under some
suitable conditions related to $L^{\theta, \infty)}(\Omega)$, which
provides a generalization of a known result due to Moscariello. A
weighted space $L^{\theta,\infty )}_w (\Omega)$ is also defined, and
the boundedness for the Hardy-Littlewood maximal operator $M_w$ and
a Calder\'on-Zygmund operator $T$ with respect to $L^{\theta,\infty
)}_w (\Omega)$ are obtained.}

{\bf AMS Subject Classification: } 46E30, 35J70.

{\bf Keywords:} Exponential class, weak monotonicity, very weak
solution, $\cal A$-harmonic equation, Hardy-Littlewood maximal
operator, Calder\'on-Zygmund operator.
\end{minipage}
\end{center}

\thispagestyle{fancyplain} \fancyhead{}
\fancyhead[L]{\textit{}\\
} \fancyfoot{} \vskip 6mm

\section{Introduction}

For $1<p<\infty$ and a bounded open subset $\Omega\subset \mbox
{R}^n$, the grand Lebesgue space $L^{p)}(\Omega)$ consists of all
functions $f(x)\in \bigcap _{0<\varepsilon \le p-1}
L^{p-\varepsilon} (\Omega)$ such that
$$
\|f\|_{p),\Omega} =\sup_{0<\varepsilon \le p-1} \left(\varepsilon
-\hspace{-4mm} \int_\Omega |f|^{p-\varepsilon} dx \right)^{\frac 1
{p-\varepsilon}} <\infty,  \eqno(1.1)
$$
where $-\hspace{-3.5mm} \int_\Omega =\frac {1}{|\Omega|}
\int_\Omega$ stands for the integral mean over $\Omega$. The grand
Sobolev space $W_0^{1,p)} (\Omega)$ consists of all functions $u\in
\bigcap _{0<\varepsilon \le p-1} W_0^{1,p-\varepsilon} (\Omega)$
such that
$$
\|u\|_{W_0^{1,p)}} =\sup _{0<\varepsilon \le p-1} \left(\varepsilon
-\hspace {-4mm} \int_\Omega |\nabla f|^{p-\varepsilon} dx
\right)^{\frac 1 {p-\varepsilon}} <\infty.  \eqno(1.2)
$$
These two spaces, slightly larger than $L^p(\Omega)$ and
$W_0^{1,p}(\Omega)$, respectively, were introduced in the paper [1]
by Iwaniec and Sbordone in 1992 where they studied the integrability
of the Jacobian under minimal hypotheses. For $p=n$ in [2] imbedding
theorems of Sobolev type were proved for functions $f\in
W_0^{1,n)}(\Omega)$. The small Lebesgue space $L^{(p}(\Omega)$ was
found by Fiorenza [3] in 2000 as the associate space of the grand
Lebesgue space $L^{p)}(\Omega)$. Fiorenza and Karadzhov gave in [4]
the following equivalent, explicit expressions for the norms of the
small and grand Lebesgue spaces, which depend only on the
non-decreasing rearrangement (provided that the underlying measure
space has measure 1):
$$
\|f\|_{L^{(p} } \approx \int_0^1 (1-\ln t)^{-\frac 1 p}
\left(\int_0^t [f^*(s)]^p ds \right)^{\frac 1 p}\frac {dt}  t, \ \
1<p<\infty,
$$
$$
\|f\|_{L^{p)}} \approx \sup _{0<t<1} (1-\ln t)^{-\frac 1 p}
\left(\int_t^1 [f^*(s)]^p ds \right)^{\frac 1 p } , \ \ 1<p<\infty.
$$
In [5], Greco, Iwaniec and Sbordone gave two more general
definitions than (1.1) and (1.2) in order to derive existence and
uniqueness results for $p$-harmonic operators. For $1<p<\infty$ and
$0\le \theta <\infty$, the grand $L^p$ space, denoted by $L^{\theta,
p)}(\Omega)$, consists of functions $f\in \bigcap _{0<\varepsilon
\le p-1} L^{p-\varepsilon} (\Omega)$ such that
$$
\|f\|_{\theta, p)} =\sup_{0<\varepsilon\le p-1} \varepsilon^ {\frac
\theta p} \|f\|_{p-\varepsilon} <\infty, \eqno(1.3)
$$
where
$$
\|f\|_{p-\varepsilon} =\left(-\hspace {-4mm} \int_\Omega
|f|^{p-\varepsilon} dx \right)^{\frac 1 {p-\varepsilon}}.
\eqno(1.4)
$$
The grand Sobolev space $W^{\theta, p)}(\Omega)$ consists of all
functions $f$ belonging to $\bigcap _{0<\varepsilon \le p-1}
W^{1,p-\varepsilon} (\Omega)$ and such that $\nabla f \in L^{\theta,
p)}(\Omega)$. That is,
$$
W^{\theta, p)}(\Omega)=\left\{f\in \bigcap _{0<\varepsilon \le p-1}
W^{1,p-\varepsilon} (\Omega): \nabla f \in L^{\theta, p)}(\Omega)
\right\}. \eqno(1.5)
$$

Grand and small Lebesgue spaces are important tools in dealing with
regularity properties for very weak solutions of $\cal A$-harmonic
equation as well as weakly quasiregular mappings, see [6, 7].

The aim of the present paper is to provide a generalization
$L^{\theta, \infty)}(\Omega)$, $0\le \theta <\infty$, of exponential
calss $EXP(\Omega)$, and prove that it is a Banach space. An
alternative definition of $L^{\theta, \infty)}(\Omega)$ is given in
terms of weak Lebesgue spaces. As an application, we obtain weak
monotonicity property for very weak solutions of $\cal A$-harmonic
equation with variable coefficients under some suitable conditions
related to $L^{\theta, \infty)}(\Omega)$. This paper also consider a
weighted space $L_w^{\theta, \infty)}(\Omega)$, and some boundedness
result for classical operators with respect to this space.

In the sequel, the letter $C$ is used for various constants, and may
change from one occurrence to another.

\section{A Generalization of Exponential Class}
Recall that $EXP(\Omega)$, the exponential class, consists of all
measurable functions $f$ such that
$$
\int_\Omega e^{\lambda |f|}dx <\infty
$$
for some $\lambda >0$. It is a Banach space under the norm
$$
\|f\|_{EXP} =\inf \left\{\lambda >0: \int_\Omega e^{|f|/\lambda} dx
\le 2 \right\}.
$$

In this section, we define a space $L^{\theta, \infty)}(\Omega)$,
$0\le \theta <\infty$, which is a generalization of $EXP(\Omega)$,
and prove that it is a Banach space.

\begin{definition} For $ \theta\ge 0$, the space $L^{\theta, \infty)
}(\Omega)$ is defined by
$$
L^{\theta, \infty)}(\Omega) =\left\{f (x) \in \bigcap _{1\le
p<\infty}L^p(\Omega): \sup_{1\le p<\infty} \frac 1
{p^\theta}\left(-\hspace {-4mm} \int_\Omega|f (x)|^pdx
\right)^{\frac 1 p}<\infty  \right\}.  \eqno(2.1)
$$
\end{definition}

It is not difficult to see that
$$
L^{\theta, \infty)}(\Omega) =\left\{g(x) \in \bigcap _{1\le
p<\infty}L^p(\Omega): {\limsup}_{p\rightarrow \infty}\frac 1
{p^\theta}\left(-\hspace {-4mm} \int_\Omega|g(x)|^pdx \right)^{\frac
1 p}<\infty  \right\}. \eqno(2.1)'
$$

There are two special cases of $L^{\theta,\infty)} (\Omega)$ that
are worth mentioning since they coincide with two known spaces.

\noindent {\bf Case 1}: $\theta =0$. In this case,
$$
L^{0, \infty)}(\Omega) =\left\{f (x) \in \bigcap _{1\le
p<\infty}L^p(\Omega): \sup_{1\le p<\infty} \left(-\hspace {-4mm}
\int_\Omega|f (x)|^pdx \right)^{\frac 1 p}<\infty  \right\}.
$$
From the fact (see [8, P12])
$$
L^\infty (\Omega) =\left\{f\in \bigcap _{1\le p <\infty}
L^p(\Omega): \lim_{p\rightarrow \infty} \|f\|_p<\infty \right\},
$$
we get $L^{0,\infty)}(\Omega) =L^{\infty}(\Omega)$.

\noindent {\bf Case 2}: $\theta =1$. The following proposition shows
that $L^{\theta, \infty}(\Omega)$ can be regarded as a
generalization of $EXP(\Omega)$.

\noindent {\bf Proposition 2.1 } {$L^{1,
\infty)}(\Omega)=EXP(\Omega)$}.

\begin{proof}
In order to realize that a function in the $L^{1,\infty)}(\Omega)$
space is in $EXP(\Omega)$, it is sufficient to read the last lines
of [2]. The vice-versa is also true, see e.g. [9, Chap. VI, exercise
no. 17].
\end{proof}

It is clear that for any $0\le \theta <\theta'\le \infty$ and any
$q<\infty$, we have the inclusions
$$
L^\infty (\Omega) \subset L^{\theta, \infty)} (\Omega) \subset
L^{\theta', \infty)}(\Omega)\subset L^q(\Omega). \eqno(2.2)
$$

The following theorem shows that, if $\theta>0$, then $L^{\theta,
\infty )}(\Omega)$ is slightly larger than $L^\infty (\Omega)$.

\begin{theorem} For $\theta >0$, the space $L^\infty (\Omega)$ is a proper subspace of
$L^{\theta, \infty )}(\Omega)$.
\end{theorem}

\begin{proof} In the proof of Theorem 2.1 we always assume
$\theta>0$. Let $f (x) \in L^{\infty}(\Omega)$, then there exists a
constant $M<\infty$, such that $|f(x)|\le M$, a.e. $\Omega$. Thus,
$$
\sup_{1\le p<\infty} \frac 1 {p^\theta} \left(-\hspace {-4mm}
\int_\Omega|f(x)|^pdx \right)^{\frac 1 p} \le \sup _{1\le p<\infty}
\frac M {p^\theta} =M<\infty,
$$
which implies $f(x)\in L^{\theta, \infty )}(\Omega)$.

The following example shows that $L^\infty (\Omega) \subset
L^{\theta, \infty )}(\Omega)$ is a proper subset. Since we have the
inclusion (2.2), then it is no loss of generality to assume that
$\theta \le 1$.  Consider the function $f(x)=(-\ln x)^\theta$
defined in the open interval $(0,1)$. It is obvious that $f(x)
\notin L^{\infty}(0,1)$. We now show that $f(x) \in L^{\theta,
\infty)}(0,1)$. In fact, for $m$ a positive integer, integration by
parts yields
$$
\begin{array}{llll}
\displaystyle \int_0^1 (-\ln x)^m dx
&=&\displaystyle \left. x(-\ln x )^m\right|_0^1 -\int_0^1 x d(-\ln x)^m \\
&=&\displaystyle -\lim_{x\rightarrow 0^+} x(-\ln x)^m +m\int_0^1
(-\ln x)^{m-1}dx.
\end{array} \eqno(2.3)
$$
By L'Hospital's Law, one has
$$
\lim_{x\rightarrow 0^+} x(-\ln x)^m =\lim_{x\rightarrow 0^+}\frac
{(-\ln x )^m}{\frac 1 x} =\lim_{x\rightarrow 0^+} \frac {m(-\ln
x)^{m-1}}{\frac 1 x } =\cdots =m!\lim_{x\rightarrow 0^+}x=0.
$$
This equality together with (2.3) yields
$$
\int_0^1 (-\ln x)^m dx=m\int_0^1 (-\ln x)^{m-1}dx.
$$
By induction,
$$
\int_0^1 f^m(x) dx =m\int_0^1 (-\ln x)^{m-1}dx=\cdots = m!\int_0^1dx
=m!.  \eqno(2.4)
$$
Recall that the function
$$
p\mapsto \left(-\hspace{-4mm} \int_\Omega |f (x)|^p dx
\right)^{\frac 1 p}
$$
is non-decreasing, thus (2.4) yields
$$
\begin{array}{llll}
&\displaystyle \sup_{1\le p<\infty}\frac 1 {p^\theta}
\left(-\hspace{-4mm}\int_0^1 |f(x)|^pdx \right)^{\frac 1 p}\\
=&\displaystyle \sup_{1\le p<\infty}\left[\frac 1 {p} \left(\int_0^1 (-\ln x)
^{p\theta}dx \right)^{\frac 1 {p\theta}}\right]^\theta\\
\le&\displaystyle \sup_{1\le p<\infty}\left[ \frac 1 {p} \left(\int_0^1
 (-\ln x)^{[p\theta]+1}dx \right)^{\frac 1 {[p\theta]+1}}\right]^\theta\\
=&\displaystyle  \sup _{1\le p<\infty}\left[ \frac {([p\theta]+1)!
^{\frac 1 {[p\theta]+1}}}{p}\right]^\theta \le \sup _{1\le p<\infty}
\left[\frac {[p\theta]+1}{p} \right]^\theta\le  2,
\end{array}
$$
where we have used the assumption $\theta \le 1$, and $[p\theta]$ is
the integer part of $p\theta$. The proof of Theorem 2.1 has been
completed.
\end{proof}

For functions $f_1(x), f_2(x) \in L^{\theta, \infty )}(\Omega)$ and
$\alpha \in \mbox {R}$, the addition $f_1(x)+f_2(x)$ and the
multiplication $\alpha f_1(x)$ are defined as usual.

\begin{theorem}
$L^{\theta, \infty)}(\Omega)$ is a linear space on $\mbox {R}$.
\end{theorem}

\begin{proof}
This theorem is easy to prove, we omit the details.
\end{proof}

For $f(x)\in L^{\theta,\infty )}(\Omega)$, we define
$$
\|f\|_{\theta, \infty ),\Omega}=\sup_{1\le p<\infty} \frac 1
{p^\theta} \left(-\hspace {-4mm} \int_\Omega|f(x)|^pdx
\right)^{\frac 1 p}. \eqno(2.5)
$$
We drop the subscript $\Omega$ from $\|\cdot\|_{\theta, \infty ),
\Omega}$ when there is no possibility of confusion.

\begin{theorem}
$\|\cdot\|_{\theta, \infty)}$ is a norm.
\end{theorem}

\begin{proof}
(1) It is obvious that $\|f\|_{\theta,\infty)}\ge 0$ and
$\|f\|_{\theta,\infty )}=0$ if and only if $f =0$ a.e. $\Omega$;

(2) For any $f_1(x), f_2(x)\in L^{\theta, \infty)} (\Omega)$,
Minkowski inequality in $L^p (\Omega)$ yields
$$
\begin{array}{llll}
\|f_1 +f_2\|_{\theta,\infty)} &=&\displaystyle \sup_{1\le p <\infty}
\frac 1 {p^\theta} \left(-\hspace{-4mm} \int_\Omega |f_1+f_2|^pdx \right)^{\frac 1 p}\\
&\le & \displaystyle \sup_{1\le p <\infty} \frac 1 {p^\theta} \left[
\left(-\hspace{-4mm} \int_\Omega |f_1|^pdx \right)^{\frac 1 {p}}+
\left(-\hspace{-4mm} \int_\Omega |f_2|^pdx \right)^{\frac 1 {p}}\right]\\
&\le & \displaystyle \sup_{1\le p <\infty}  \frac 1 {p^\theta}
\left(-\hspace{-4mm} \int_\Omega |f_1|^pdx \right)^{\frac 1 {p}}+
\sup_{1\le p <\infty}  \frac 1 {p^\theta}
\left(-\hspace{-4mm} \int_\Omega |f_2|^pdx \right)^{\frac 1 {p}}\\
&= & \displaystyle \|f_1\|_{\theta, \infty)}
+\|f_2\|_{\theta,\infty)};
\end{array}
$$

(3) For all $\lambda \in \mbox {R}$ and all $f(x)\in L^{\theta,
\infty)}(\Omega)$, it is obvious that $\|\lambda f
\|_{\theta,\infty)}=|\lambda| \|f\|_{\theta,\infty)}$.
\end{proof}

\begin{theorem}
$\left(L^{\theta, \infty)}(\Omega), \|\cdot\|_{\theta,
\infty)}\right)$ is a Banach space.
\end{theorem}

\begin{proof}
Suppose that $\{f_n\}_{n=1}^\infty \subset L^{\theta, \infty)}
(\Omega)$, and for any positive integer $p$,
$$
\|f_{n+p} -f_n\|_{\theta, \infty)} \rightarrow 0, \ \ n \rightarrow
\infty. \eqno(2.6)
$$
Since $\Omega$ is $\sigma$-finite, then $\Omega =\bigcup
_{m=1}^\infty \Omega_m$ with $|\Omega_m|<\infty$. It is no loss of
generality to assume that the $\Omega_m$s are disjoint. (2.4)
implies that for any positive integer $p$,
$$
\int_{\Omega _m} |f _{n+p}(x) -f_n (x)|dx \rightarrow 0, \ \
n\rightarrow \infty.
$$
Thus, by the completeness of $L^1(\Omega _m)$, there exists
$f^{(m)}(x)\in L^1(\Omega_m)$, such that
$$
f_n(x) \rightarrow f^{(m)} (x), \ n\rightarrow \infty, \ \mbox { in
} L^1(\Omega_m). \eqno(2.7)
$$
Hence for any positive integer $m$, there exists a subsequence
$\{f_n^{(m)} (x)\}$ of $\{f^{m-1}_n(x)\}$,
$\{f_n^{(0)}(x)\}=\{f_n(x) \}$, such that
$$
f _n^{(m)} (x) \rightarrow f^{(m)} (x), \ \ n\rightarrow \infty, \
\mbox { a.e. } x\in \Omega _m.
$$
If we let
$$
f (x) =f^{(m)} (x), \ \ x\in \Omega_m, \ \ m=1,2,\cdots ,
$$
then
$$
f _n^{(n)}(x) \rightarrow f (x), \ \ n\rightarrow \infty, \mbox {
a.e. } x\in \Omega.
$$
It is no loss of generality to assume that the subsequence $\{f_n
^{(n)}(x)\}$ of $\{f _n(x)\}$ is itself, thus
$$
f_n(x) \rightarrow f (x), \ \ n\rightarrow \infty, \ \mbox { a.e. }
x\in \Omega.
$$
We now prove $f (x) \in L^{\theta, \infty)} (\Omega)$ and $\|f_n-f
\|_{\theta, \infty)}\rightarrow 0$, $(n\rightarrow \infty)$. In
fact, by (2.6), for any $\varepsilon >0$, there exists
$N=N(\varepsilon)$, such that if $n>N$, then
$$
\sup_{1\le q<\infty} \frac 1 {q^\theta} \left(-\hspace{-4mm}
\int_\Omega |f _{n+p}(x) -f_n(x)| ^q dx \right)^{\frac 1 q}
<\varepsilon.
$$
Let $p\rightarrow \infty$, one has
$$
\sup_{1\le q<\infty} \frac 1 {q^\theta} \left(-\hspace{-4mm}
\int_\Omega |f_{n}(x) -f(x)| ^q dx \right)^{\frac 1 q} <\varepsilon,
\ \ n>N.
$$
Hence $f (x) \in L^{\theta, \infty)}(\Omega)$, and $\|f _n(x) -f
(x)\|_{\theta, \infty)} \rightarrow 0$, $n\rightarrow \infty$. This
completes the proof of Theorem 2.4.
\end{proof}

\begin{definition} The grand Sobolev space $W^{\theta, \infty)}(\Omega)$
consists of all functions $f$ belonging to $\bigcap _{1\le
p<\infty}W^{1,\infty)}(\Omega)$ and such that $\nabla f \in
L^{\theta, \infty)}(\Omega)$. That is,
$$
W^{\theta, \infty)}(\Omega) =\left\{f\in \bigcap _{1\le
p<\infty}W^{1,\infty)}(\Omega): \nabla f \in L^{\theta,
\infty)}(\Omega) \right\}.
$$
\end{definition}

This definition will be used in Section 4.

\section{An Alternative Definition of $L^{\theta, \infty}(\Omega)$}
In this section, we give an alternative definition of $L^{\theta,
\infty)}(\Omega)$ in terms of weak Lebesgue spaces. Let us first
recall the definition of weak $L^p (0<p<\infty)$ spaces, or the
Marcinkiewicz spaces, $L_{weak}^p(\Omega)$, see [10, Chapter 1,
Section 2], [11, Chapter 2, Section 5] or [12, Chapter 2, Section
18].

\begin{definition}
Let $0<p<\infty$. We say that $f\in L_{weak}^{p}(\Omega)$ if and
only if there exists a positive constant $k=k(f)$ such that
$$
f_*(t)= |\{x\in \Omega: |f(x)|>t\}| \le \frac {k}{t^p} \eqno(3.1)
$$
for every $t>0$, where $|E|$ is the $n$-dimensional Lebesgue measure
of $E\subset \mbox {R}^n$, and $f_* (t)=|\{x\in \Omega: |f(x)|>t\}|$
denotes the distribution function of $f$.
\end{definition}

For $p>1$, we recall that if $f\in L_{weak}^p(\Omega)$, then $f\in
L^q (\Omega)$ for every $1\le q <p$, and $f\in L_{weak}^p(\Omega)$
if and only if for every measurable set $E\subset \Omega$, the
following inequality holds
$$
\int_E |f(x)| d x \le c|E|^{\frac {p-1}{p}}
$$
for some constant $c>0$.

(3.1) is equivalent to
$$
M_p(f)=\left[\frac 1 {|\Omega|} \sup_{t>0} t^p f_*(t) \right]^{\frac
1 p}<\infty. \eqno(3.2)
$$
Recall also that
$$
\int_\Omega |f(x)|^s dx =s\int_0^\infty t^{s-1} f_*(t)dt <\infty.
\eqno(3.3)
$$

\begin{definition} For $\theta\ge 0$, the weak space $L_{weak}^{\theta, \infty
}(\Omega)$ is defined by
$$
L_{weak}^{\theta, \infty} (\Omega) =\left\{f\in \bigcap _{1\le p
<\infty} L_{weak} ^p (\Omega): \sup _{1\le p <\infty}\frac {M_p(f)}
{p^\theta} <\infty\right\}. \eqno(3.5)
$$
\end{definition}

The following theorem shows that $L_{weak}^{\theta,
\infty}(\Omega)=L^{\theta, \infty)}(\Omega)$, thus
$L_{weak}^{\theta, \infty} (\Omega)$ can be regarded as an
alternative definition of the space $L^{\theta, \infty)} (\Omega)$.

\begin{theorem}
$$
L_{weak}^{\theta, \infty}(\Omega)=L^{\theta, \infty) }(\Omega).
$$
\end{theorem}
\begin{proof} We divided the proof into two steps.

{\bf Step 1 } $L_{weak}^{\theta, \infty}(\Omega)\subset L^{\theta,
\infty) }(\Omega)$.

If $1\le s <p$, for each $a>0$, one can split the integral in the
right-hand side of (3.3) to obtain
$$
\begin{array}{llll}
\displaystyle \int_\Omega |f|^s dx &=&\displaystyle  s\int_0^a
t^{s-1} f_*(t)dt + s\int_a^\infty t^{s-1} f_*(t)dx \\
& \le &\displaystyle |\Omega |a^s +\frac {sa^{s-p}}{p-s} |\Omega|
M_p^p(f).
\end{array}
$$
The second integral has been estimated by the inequality $f_*(t) \le
|\Omega| t^{-p} M_p^p (f)$, which is a direct consequence of the
definition of the constant $M_p(f)$ (see (3.2)). Setting $a=M_p(f)$
we arrive at
$$
-\hspace{-4mm} \int_\Omega |f|^s dx \le M_p^s(f) +\frac s {p-s}
M_p^s(f) =\frac {p}{p-s} M_p(f).
$$
This implies
$$
\frac 1 {s^\theta} \left(-\hspace {-4mm} \int_\Omega |f|^s dx
\right)^{\frac 1 s } \le \frac 1 {s^\theta} \left(\frac
{p}{p-s}\right)^{\frac 1 s } M_p(f). \eqno(3.6)
$$
Therefore
$$
\begin{array}{llll}
&\displaystyle \sup_{1\le s <\infty} \frac 1 {s^\theta}
\left(-\hspace{-4mm}
\int_\Omega |f|^s dx \right)^{\frac 1 s } \\
=&\displaystyle \max \left\{ \sup_{1\le s <2} \frac 1 {s^\theta}
\left(-\hspace{-4mm} \int_\Omega |f|^s dx \right)^{\frac 1 s },
\sup_{2\le s <\infty }  \frac 1 {s^\theta} \left(-\hspace{-4mm}
\int_\Omega |f|^s dx \right)^{\frac 1 s }\right\}\\
\le &\displaystyle \max \left\{\|f\|_2, \sup_{2\le s =p-1<\infty}
\frac 1 {s^\theta} \left(-\hspace{-4mm} \int_\Omega |f|^s dx
\right)^{\frac 1
s }\right\}\\
\le &\displaystyle \max \left\{\|f\|_2, \sup_{2\le s<\infty} \frac 1
{s^\theta} (s+1)^{\frac 1 s } M_{s+1} (f)\right\}\\
\le &\displaystyle \max \left\{\|f\|_2, 4\sup_{1\le s<\infty} \frac
{M_s(f)}{{s^\theta}}\right\}<\infty,
\end{array}
$$
here we have used (3.6) and the definition of $L_{weak}^\infty
(\Omega)$.

{\bf Step 2 } $L^{\theta, \infty) }(\Omega)\subset
L_{weak}^{\infty}(\Omega)$.

Since for any $t>0$,
$$
t^p f_*(t) =t^p \int_{\{x\in \Omega: |f(x)|>t\}} dx \le \int_{\{x\in
\Omega: |f(x)|>t\} }|f|^p dx \le \int_\Omega |f|^p dx,
$$
then
$$
\sup _{t>0} t^pf_*(t) \le \int_\Omega |f|^p dx .
$$
This implies
$$
M_p(f) =\left[\frac 1 {|\Omega|}  \sup_{t>0} t^pf_*(t)\right]^{\frac
1 p} \le \left(-\hspace{-4mm} \int_\Omega |f|^p dx \right)^{\frac 1
p }.
$$
Thus
$$
\sup_{1\le p <\infty} \frac {M_p(f)}{p^\theta} \le \sup_{1\le p
<\infty}\frac 1 {p^\theta}\left(-\hspace{-4mm} \int_\Omega |f|^p dx
\right)^{\frac 1 p }<\infty.
$$
The proof of Theorem 3.1 has been completed.
\end{proof}

\section{An Application}

In this section, we give an application of the space $L^{\theta,
\infty} (\Omega)$ to monotonicity property of very weak solutions of
the $\cal A$-harmonic equation
$$
\mbox {div}{\cal A}(x,\nabla u(x)) =0,  \eqno(4.1)
$$
where ${\cal A}: \Omega \times \mbox {R}^n \rightarrow \mbox {R}^n$
be a mapping satisfying the following assumptions:

(1) the mapping $x\mapsto {\cal A}(x,\xi)$  is measurable for all
$\xi \in \mbox {R}^n$,

(2) the mapping $\xi \mapsto {\cal A}(x,\xi)$ is continuous for a.e. $x\in \mbox {R}^n$,\\
for all $\xi \in \mbox {R}^n$, and a.e. $x\in \mbox {R}^n$,

(3)
$$
\langle {\cal A}(x,\xi), \xi \rangle \ge \gamma (x) |\xi|^p,
$$

(4)
$$
|{\cal A}(x,\xi)|\le \tau (x) |\xi|^{p-1},
$$
where $1<p<\infty$, $0<\gamma (x) \le \tau (x)<\infty$, a.e.
$\Omega$.

Conditions (1) and (2) insure that the composed mapping $x\mapsto
{\cal A}(x,g(x))$ is measurable whenever $g$ is measurable. The
degenerate ellipticity of the equation is described by condition
(3). Finally, condition (4) guarantees that, for any $0\le \theta
<\infty$ and any $\varepsilon>0$, ${\cal A}(x,\nabla u)$ can be
integrated for $u\in W^{\theta,p}(\Omega)$ against functions in
$W^{1,\frac {p-\varepsilon}{1-p\varepsilon}}(\Omega) $ with compact
support.

\begin{definition} A function
$u\in W_{loc}^{1,r} (\Omega)$, $\max \{1, p-1\} <r\le p$, is called
a very weak solution of (4.1), if
$$
\int_\Omega \langle {\cal A}(x,\nabla u(x)), \nabla \varphi (x)
\rangle dx =0
$$
for all $\varphi \in W_0^{1,\frac {r}{r-p+1}}(\Omega)$.
\end{definition}

A fruitful idea in dealing with the continuity properties of Sobolev
functions is the notion of monotonicity. In one dimension a function
$u:\Omega \rightarrow \mbox {R}$  is monotone if it satisfies both a
maximum and minimum principle on every subinterval. Equivalently, we
have the oscillation bounds $\mbox {osc}_I u \le \mbox
{osc}_{\partial I} u$ for every interval $I\subset \Omega$. The
definition of monotonicity in higher dimensions closely follows this
observation.

A continuous function $u:\Omega \rightarrow \mbox {R}^n$  defined in
a domain $\Omega \subset \mbox {R}^n$ is monotone if
$$
\mbox {osc}_B u \le \mbox {osc}_{\partial B} u
$$
for every ball $B\subset \mbox {R}^n$. This definition in fact goes
back to Lebesgue [13] in 1907 where he first showed the relevance of
the notion of monotonicity in the study of elliptic PDEs in the
plane. In order to handle very weak solutions of $\cal A$-harmonic
equation, we need to extend this concept, dropping the assumption of
continuity. The following definition can be found in [14], see also
[6, 7].

\begin{definition}
A real-valued function $u\in W_{loc}^{1,1}(\Omega)$ is said to be
weakly monotone if, for every ball $B\subset \Omega$ and all
constants $m\le M$ such that
$$
|M-u| -|u-m| +2u -m-M \in W_0^{1,1}(B),  \eqno(4.2)
$$
we have
$$
m\le u(x) \le M  \eqno(4.3)
$$
for almost every $x\in B$.
\end{definition}

For continuous functions (4.2) holds if and only if $m\le u(x)\le M$
on $\partial B$. Then (4.3) says we want the same condition in $B$,
that is the maximum and minimum principles.

Manfredi's paper [14] should be mentioned as the beginning of the
systematic study of weakly monotone functions. Koskela, Manfredi and
Villamor obtained in [15] that $\cal A$-harmonic functions are
weakly monotone. In [16], the first author obtained a result which
states that very weak solutions $u\in W_{loc}
^{1,p-\varepsilon}(\Omega)$ of the $\cal A$-harmonic equation are
weakly monotone provided $\varepsilon $ is small enough. The
objective of this section is to extend the operator $\cal A$ to
spaces slightly larger than $L^p(\Omega)$.

\begin{theorem}
Let $\gamma (x)>0$, a.e. $\Omega$, $\tau (x)\in L^{\theta_1,
\infty)} (\Omega)$. If $u\in W^{\theta_2,p)}(\Omega)$ is a very weak
solution to (4.1), then it is weakly monotone in $\Omega$ provided
that $\theta_1+\theta_2<1$.
\end{theorem}

\begin{proof}
For any ball $B\subset \Omega$ and $0<\varepsilon <1$, let
$$
\psi =(u-M)^+ -(m-u)^+ \in W_0^{1,p-\varepsilon} (B).
$$
It is obvious that
$$
\nabla \psi =\left\{
\begin{array}{llll}
0, & \mbox { for } m\le u(x) \le M,\\
\nabla u, & \mbox { otherwise, say, on a set } E\subset B.
\end{array}
\right.
$$
Consider the Hodge decomposition (see [6]),
$$
|\nabla \psi|^{-p\varepsilon} \nabla \psi =\nabla \varphi +h.
$$
The following estimate holds
$$
\|h\|_{\frac {p-\varepsilon}{1-p\varepsilon}} \le C \varepsilon
\|\nabla \psi\|_{p-\varepsilon}^{1-p\varepsilon}.  \eqno(4.4)
$$
Definition 4.1 with $\varphi $ acting as a test function yields
$$
\int_E \langle {\cal A}(x,\nabla u), |\nabla u|^{-p\varepsilon}
\nabla u \rangle dx =\int_E \langle {\cal A}(x,\nabla u), h \rangle
 dx. \eqno(4.5)
$$
H\"older's inequality together with the conditions (3), (4), (4.4)
and (4.5) yields
$$
\begin{array}{llll}
&\displaystyle \int_E \gamma (x) |\nabla u|^{p(1-\varepsilon)} dx\\
\le &\displaystyle \int_E \langle {\cal A}(x,\nabla u), |\nabla
u|^{-p\varepsilon}
\nabla u \rangle dx \\
=&\displaystyle \int_E \langle {\cal A}(x,\nabla u),h\rangle dx \\
\le &\displaystyle \int_E \tau (x) |\nabla u|^{p-1} |h|dx \\
\le &\displaystyle \|\tau\|_{\frac
{p-\varepsilon}{(p-1)\varepsilon}} \|\nabla
u\|_{p-\varepsilon}^{p-1} \|h\|_{\frac
{p-\varepsilon}{1-p\varepsilon}} \\
\le &\displaystyle C\varepsilon \|\tau\|_{\frac
{p-\varepsilon}{(p-1)\varepsilon}} \|\nabla
u\|_{p-\varepsilon}^{p(1-\varepsilon)}\\
=&\displaystyle C |E|\varepsilon \cdot \varepsilon ^{-\theta_2
(1-\varepsilon )} \left[\frac {p-\varepsilon }{(p-1)\varepsilon
}\right]^{\theta_1} \left[\frac {(p-1)\varepsilon }{p-\varepsilon
}\right]^{\theta_1} \left(-\hspace {-4mm} \int_E |\tau|^{\frac
{p-\varepsilon }{(p-1)\varepsilon }} dx \right)^{\frac
{(p-1)\varepsilon }{p-\varepsilon }}\times\\
&\displaystyle \times \varepsilon ^{\theta_2 (1-\varepsilon )}
\left(-\hspace {-4mm} \int_E |\nabla u|^{p-\varepsilon } \right)
^{\frac {p(1-\varepsilon )}{p-\varepsilon }}.
\end{array} \eqno(4.6)
$$
The condition $\tau \in L^{\theta_1, \infty) } (\Omega)$ implies
$$
\lim_{\varepsilon \rightarrow 0^+} \left[\frac {(p-1)\varepsilon
}{p-\varepsilon }\right]^{\theta_1} \left(-\hspace {-4mm} \int_E
|\tau|^{\frac {p-\varepsilon }{(p-1)\varepsilon }} dx \right)^{\frac
{(p-1)\varepsilon }{p-\varepsilon }}\le \|\tau \|_{\theta_1,
\infty)}<\infty.  \eqno(4.7)
$$
Since $u\in W^{\theta_2,p)}(\Omega)$, then
$$
\lim_{\varepsilon \rightarrow 0^+}\varepsilon ^{\theta_2
(1-\varepsilon )} \left(-\hspace {-4mm} \int_E |\nabla
u|^{p-\varepsilon } \right) ^{\frac {p(1-\varepsilon
)}{p-\varepsilon }}\le \|\nabla u\|_{\theta_2, p)}^p <\infty,
\eqno(4.8)
$$
By $\theta_1 +\theta_2<1$, we have
$$
\lim_{\varepsilon \rightarrow 0^+} \varepsilon \cdot \varepsilon
^{-\theta_2 (1-\varepsilon )} \left[\frac {p-\varepsilon
}{(p-1)\varepsilon }\right]^{\theta_1}=\left(\frac p
{p-1}\right)^{\theta_1} \lim_{\varepsilon \rightarrow 0^+}
\varepsilon ^{1-\theta_2 (1-\varepsilon )-\theta_1} =0.  \eqno(4.9)
$$
Combining (4.6)-(4.9), and taking into account the assumption
$\gamma (x)>0$, a.e. $\Omega$, we arrive at $\nabla u=0$, a.e. $E$.
This implies that $(u-M)^+ -(m-u)^+$ vanishes a.e. in $B$, and thus
$(u-M)^+ -(m-u)^+$ must be the zero function in $B$, completing the
proof of Theorem 4.1.
\end{proof}

\begin{rem}
We remark that the result in Theorem 4.1 is a generalization of a
result  due to Moscariello, see [17, Corollary 4.1].
\end{rem}

\section{A Weighted Version}
A weight is a locally integrable function on $\mbox {R}^n$ which
takes values in $(0,\infty)$ almost everywhere. For a weight $w$ and
a measurable set $E$, we define $w(E)=\int_E w(x)dx$ and the
Lebesgue measure of $E$ by $|E|$. The weighted Lebesgue spaces with
respect to the measure $w(x)dx$ are denoted by $L^p_w$ with
$0<p<\infty$. Given a weight $w$, we say that $w$ satisfies the
doubling condition if there exists a constant $C>0$ such that for
any cube $Q$, we have $w(2Q)\le Cw(Q)$, where $2Q$ denotes the cube
with the same center as $Q$ whose side length is 2 times that of
$Q$. When $w$ satisfies this condition, we denote $w\in \Delta_2$,
for short.

A weight function $w$ is in the Muckenhoupt class $A_p$ with
$1<p<\infty$ if there exists $C>1$ such that for any cube $Q$
$$
\left(-\hspace {-4mm} \int_Q w (x)dx \right)\left(-\hspace {-4mm}
\int_Q w (x)^{1-p'}dx \right)^{p-1} \le C, \eqno(5.1)
$$
where $\frac 1 p +\frac 1 {p'}=1$. We define $A_\infty =\bigcup
_{1<p<\infty}A_p$.

Let $w$ be a weight. The Hardy-Littlewood maximal operator with
respect to the measure $w(x)dx$ is defined by
$$
M_wf(x) =\sup_{Q\ni x} \frac 1 {w(Q)} \int_Q |f(x)|w(x)dx.
$$
We say that $T$ is a Calder\'on-Zygmund operator if there exists a
function $K$ which satisfies the following conditions:
$$
Tf(x) =\mbox {p.v.} \int_{\mbox {R}^n} K(x-y) f(y) dy.
$$
$$
|K(x)| \le \frac {C_K}{|x|^n} \mbox { and } |\nabla K(x)|\le \frac
{C_K}{ |x|^{n+1}}, \  \ x\ne 0.
$$

For $w$ a weight and $0\le \theta <\infty$, we define the space
$L_w^{\theta, \infty)} (\Omega)$ as follows
$$
L^{\theta, \infty)}_w (\Omega)=\left\{f(x) \in \bigcap _{1<p<\infty}
L^p_w (\Omega): \|f\|_{L^{\theta, \infty)}_w (\Omega)}
<\infty\right\},
$$
where
$$
\|f\|_{L^{\theta, \infty)}_w(\Omega)} =\sup_{1<p<\infty} \frac 1
{p^\theta} \left( \frac 1 {w(\Omega)} \int_\Omega |f(x)|^p w(x) dx
\right)^{\frac 1 p}.
$$
The following lemma comes from [18].

\begin{lemma}
If $1<p<\infty$ and $w\in \Delta _2$, then the operator $M_w$ is
bounded on $L^p_w(\Omega)$.
\end{lemma}

\begin{theorem}
The operator $M_w$ is bounded on $L^{\theta, \infty)}_w (\Omega)$
for $0\le \theta <\infty$ and $w\in \Delta _2$.
\end{theorem}
\begin{proof}
By Lemma 5.1, since for $1<p<\infty$ and $w\in \Delta_2$, the
operator $M_w$ is bounded on $L^p_w (\Omega)$, then
$$
\left( \int_\Omega |M_w f(x)|^p w(x)dx \right)^{\frac 1 p} \le C
\left( \int_\Omega |f(x)|^p w(x)dx \right)^{\frac 1 p}.
$$
This implies
$$
\begin{array}{llll}
\displaystyle \|M_w f\|_{L^{\theta,\infty)}_w (\Omega)}
&\displaystyle =\sup _{1<p<\infty} \frac 1 {p^\theta} \left(\frac 1
{w(\Omega)} \int_\Omega |M_w f(x)|^p
w(x)dx \right)^{\frac 1 p } \\
&\displaystyle \le C\sup _{1<p<\infty} \frac 1 {p^\theta}
\left(\frac 1 {w(\Omega)} \int_\Omega |f(x)|^p w(x)dx \right)^{\frac
1 p } =\|f\|_{L^{\theta,\infty)}_w (\Omega)},
\end{array}
$$
completing the proof of Theorem 5.1
\end{proof}

The following lemma can be found in [19].

\begin{lemma}
If $w\in A_\infty$, then there exists $q\in (1,\infty)$ such that
$w\in A_q$.
\end{lemma}

The following lemma can be found in [20, 21].

\begin{lemma}
If $1<p<\infty$ and $w\in A_p$, then a Calder\'on-Zygmund operator
$T$ is bounded on $L^p_w (\Omega)$.
\end{lemma}

\begin{theorem}
A Calder\'on-Zygmund operator $T$ is bounded on $L^{\theta,
\infty)}_w (\Omega)$ for $0\le \theta <\infty$ and $w\in A_\infty$.
\end{theorem}
\begin{proof}
By $w\in A_\infty$ and Lemma 5.2, one has $w\in A_q$ for some $q\in
(1,\infty)$. For $1<p<q<\infty$, H\"older's inequality yields
$$
\begin{array}{llll}
&\displaystyle \int_\Omega |Tf(x)|^p w(x)dx =\int_\Omega |Tf(x)|^p
w(x)^{\frac p q} w(x)^{\frac {q-p}{p}}dx \\
\le &\displaystyle \left( \int_\Omega |Tf(x)|^q w(x)dx\right)^{\frac
p q}  \left(\int_\Omega w(x)dx\right)^{\frac {q-p} q}.
\end{array}
$$
Thus
$$
\begin{array}{llll}
&\displaystyle \frac 1 {p^\theta} \left(\frac 1 {w(\Omega)}\int_
\Omega |Tf(x)|^p w(x)dx \right)^{\frac 1 p }\\
\le &\displaystyle \frac 1 {p^\theta} \left(\frac 1 {w(\Omega)}\int_
\Omega |Tf(x)|^q w(x)dx \right)^{\frac p q} \left(\frac 1
{w(\Omega)}\int_\Omega w(x)dx \right)^{\frac {q-p} q}\\
= &\displaystyle \frac 1 {p^\theta} \left(\frac 1 {w(\Omega)}\int_
\Omega |Tf(x)|^q w(x)dx \right)^{\frac p q} .
\end{array}  \eqno(5.1)
$$
Lemma 5.3 yields
$$
\begin{array}{llll}
&\displaystyle \|Tf\|_{L^{\theta, \infty)}_w (\Omega)} \\
=&\displaystyle \max \left\{ \sup_{1<p<q} \frac {1}{p^\theta}
\left(\frac 1 {w(\Omega)} \int_\Omega |Tf(x)|^p w(x) dx
\right)^{\frac 1 p }, \sup_{q\le p<\infty} \frac {1}{p^\theta}
\left(\frac 1 {w(\Omega)} \int_\Omega
|Tf(x)|^p w(x) dx \right)^{\frac 1 p } \right\}\\
=&\displaystyle  \max \left\{ \sup_{1<p<q} \frac {1}{p^\theta}
\left(\frac 1 {w(\Omega)} \int_\Omega |Tf(x)|^q w(x) dx
\right)^{\frac p q}, \sup_{q\le p<\infty} \frac {1}{p^\theta}
\left(\frac 1 {w(\Omega)} \int_\Omega
|Tf(x)|^p w(x) dx \right)^{\frac 1 p } \right\}\\
\le &\displaystyle \max \left\{ \sup_{1<p<q}\left(\frac q p
\right)^\theta, 1 \right\} \sup_{q\le p<\infty} \frac {1}{p^\theta}
\left(\frac 1 {w(\Omega)} \int_\Omega |Tf(x)|^p w(x) dx
\right)^{\frac 1 p }\\
\le &\displaystyle Cq^\theta \sup_{q\le p<\infty} \frac
{1}{p^\theta} \left(\frac 1 {w(\Omega)} \int_\Omega |Tf(x)|^p w(x)
dx\right)^{\frac 1 p }\\
\le &\displaystyle Cq^\theta \|f\|_{L^{\theta, \infty)}_w (\Omega)}.
\end{array}
$$
As desired.
\end{proof}

\vspace{4mm}

\noindent {\bf Acknowledgement } This study was funded by NSFC
(10971224) and NSF of Hebei Province (A2011201011).

\vspace{8mm}

\rm \footnotesize
\linespread{2}

\end{document}